\documentclass[11pt,reqno]{amsart}

\usepackage{amssymb,amsthm,amsmath}

\numberwithin{equation}{section}

\theoremstyle{plain}
\newtheorem{prop}{Proposition}[section]
\newtheorem{lem}[prop]{Lemma}

\newtheorem{thm}[prop]{Theorem}

\theoremstyle{definition}

\theoremstyle{remark}

\DeclareMathOperator{\tr}{tr}

\newcommand\N{\mathbb{N}}

\newcommand\R{\mathbb{R}}
\newcommand\Z{\mathbb{Z}}

\newcommand\T{\mathbb{T}}

\begin{document}
\title[]{Effective equidistribution of periodic orbits for subshifts of finite type}

\author{Shirali\@ Kadyrov}

\address{Department of Mathematics,
Nazarbayev University,
Astana, Kazakhstan}

\email{shirali.kadyrov@nu.edu.kz}

\begin{abstract}
We study equidistribution of certain subsets of periodic orbits for subshifts of finite type. Our results solely rely on the growth of these subsets. As a consequence, effective equidistribution results are obtained for both hyperbolic diffeomorphisms and expanding maps on compact manifolds.
\end{abstract}

\maketitle

\section{Introduction}\label{sec:intro}

For a given $s \in \N$ and an $s \times s$ transition matrix $A$ with entries zero or one we let $(\Sigma_A^+,\sigma)$ denote the one-sided subshift of finite type where $\Sigma_A^+$ is the symbolic space given by
$$\Sigma_A^+=\{x=(x_n)_{n \ge 0} \in \prod_{n=0}^\infty \{1,2,\dots,s\}: A(x_n,x_{n+1})=1, \forall n \in \N\},$$
and $\sigma:\Sigma_A^+ \to \Sigma_A^+$ is the shift map by $(\sigma(x))_{n}=x_{n+1}.$
For a given $\theta \in (0,1)$ define a metric $d_\theta$ on $\Sigma_A^+$ by $d_\theta(x,y)=\theta^{t(x,y)} \text{ where } t(x,y)=\max\{n \ge 0: x_i=y_i, 0\le i<n \}.$ We similarly define two-sided subshift of finite type $(\Sigma_A,\sigma)$ where 
$$\Sigma_A=\{x=(x_n)_{n \ge 0} \in \prod_{n=-\infty}^\infty \{1,2,\dots,s\}: A(x_n,x_{n+1})=1, \forall n \in \Z\}$$ and the metric is given by $d_\theta(x,y)=\theta^{t(x,y)} \text{ where } t(x,y)=\max\{n \ge 0: x_i=y_i, |i|<n \}.$ Also, for any continuous function $g$ on $\Sigma_A^+$ we let 
$$|g|_\theta=\sup_{n \ge 0} \left\{\frac{|g(x)-g(y)|}{\theta^n}: x_i=y_i, 0 \le i\le n\right\}.$$ 
We similarly define $|\cdot|_\theta$ on $\Sigma_A$ with $0 \le i \le n$ replaced by $|i| \le n.$
In particular, $|g|_\theta <\infty$ implies that $g$ is a Lipschitz function with the least Lipschitz constant $|g|_\theta.$ Consider a norm $\|\cdot\|_\theta=|\cdot |_\theta+|\cdot|_\infty$ where $|\cdot|_\infty$ is the supremum norm and let $\mathcal F_\theta^+$ denote the space of all continuous functions $f$ on $\Sigma_A^+$ with $\|f\|_\theta<\infty.$ Analogously we define $\mathcal F_\theta$ on $\Sigma_A$. For both $(\Sigma_A^+,\sigma)$ and $(\Sigma_A,\sigma)$ we let $h(\sigma)$ denote the topological entropy and $m$ be the measure of maximal entropy so that $h(\sigma)=h_m(\sigma).$ Let $\xi=\{C_1, C_2,\dots,C_s\}$ denote the generating partition of $\Sigma_A^+$ (or $\Sigma_A$), where $C_i=\{x : x_0=i\}.$ To simplify the notation we let $\xi_{\ell}^n:=\bigvee_{i=\ell}^{n} \sigma^{-i} \xi$. In \cite{Ka14} we studied the effective uniqueness of $m$, the measure of maximal entropy. By effective uniqueness we mean a statement that gives how close a given measure to the measure of maximal entropy if its metric entropy is close to maximal entropy. See results from \cite{Polo, Ruh13} similar to  \cite{Ka14}. In this paper we obtain the following improvement of \cite[Theorem~1.1]{Ka14}, which will lead to effective equidistribution statements. The matrix $A$ is said to be \emph{irreducible} if  for each pair $(i,j)$ there exists $n\ge 1$ such that $A^n(i,j)>0$. We say that $A$ is \emph{aperiodic} if $A(i,i)=1$ for all $i=1,2,\dots,s$.

\begin{thm}\label{thm:main}
Assume that $A$ is irreducible and aperiodic. Then, there exists a constant $c>0$ such that for any $N \in \N \cup \{\infty\}$ and $\sigma$-invariant probability measure $\mu$ on $\Sigma_A^+$ and any Lipschitz function $f$ we have
\begin{equation}\label{eqn:main}
\left| \int f d\mu - \int f d  m\right| \le  c \|f\|_\theta \left( \theta^{\frac{N}2}+2\sqrt{2}\left(h(\sigma)-\frac1N H_\mu( \xi_0^{N-1})\right)^{\frac12}\right),
\end{equation}
where $m$ is the measure of maximal entropy on $\Sigma_A^+.$ Moreover, the same result holds for two-sided subshift $(\Sigma_A,\sigma)$ with the exponent $N/2$ of $\theta$ replaced by $N/4$.
\end{thm}

This generalizes and improves \cite[Theorem~4.1.2 and Theorem~4.1.3]{Polo}. 

We now want to discuss how Theorem~\ref{thm:main} can be applied to show the effective equidistribution of periodic orbits. In other words, we want to obtain a rate of convergence of the distribution of periodic orbits. For any $n \in \N$ we let $\text{Fix}_n$ denote the set of periodic points of period $n$, namely
$${\rm Fix}_n=\{x \in \Sigma_A^+ : \sigma^n(x)=x\}.$$
By abuse of notation we let ${\rm Fix}_n$ also denote the periodic orbits of $\Sigma_A$ of order $n$. For any nonempty finite set $I$ in $\Sigma_A^+$  (or $\Sigma_A$) we let $\mu_I$ denote the uniform probability measure supported on $I$, namely, 
$$\mu_I =\frac1{|I|} \sum_{x \in I} \delta_x.$$ 
Clearly each nonempty element of $ \xi_0^{n-1}$ contains exactly one element from ${\rm Fix}_n$. More precisely, for any $x \in {\rm Fix}_n$ we have  
$$x \in P(x,n):=C_{x_0} \cap \sigma^{-1} C_{x_1} \cap \cdots \cap \sigma^{-(n-1)} C_{x_{n-1}} \in \xi_0^{n-1}.$$
Thus, for any $P \in \xi_0^{n-1}$ we have $\mu_I(P)=\frac{1}{|I|}$ if $P=P(x,n)$ for some $x \in I$ and it is zero otherwise. Thus,
$$H_{\mu_I}(\xi_0^{n-1})= \log |I|.$$
Consequently, applying Theorem~\ref{thm:main} we see that for any Lipschitz $f$ it holds
$$\left| \int f d\mu - \int f d  m\right| \le c \|f\|_\theta \left( \theta^{\frac{n}4}+2\sqrt{2}\left(h_{  m}(\sigma)-\frac1n \log |I|  \right)^{\frac12}\right),$$
this proves the following.

\begin{thm}\label{thm:mainPO}
Fix $n \in \N$ and let $I$ be a nonempty invariant subset of ${\rm Fix}_n$. Then, there exists $c>0$ such that for any Lipschitz function $f$ we have
$$|\int f\, dm - \int f\, d\mu_{I}| \le c \|f\|_\theta \left( \theta^{\frac{n}4}+2\sqrt{2}\sqrt{h(\sigma) - \frac{1}n\log |I|}\right).$$
\end{thm}

As an immediate consequence we get

\begin{thm}\label{thm:seqI}
If $\{I_n\}$ is a sequence of invariant sets with $I_n \subset {\rm Fix}_n$ and $\varphi(n):=h(\sigma)-\frac1n \log|I_n| \to 0$ as $n \to \infty$ then there exists a constant $c >0$ such that for any $n \in \N$ and Lipschitz function $f$ we have
$$|\int f\, dm - \int f\, d\mu_{I_n}| \le c \|f\|_\theta \left( \theta^{\frac{n}4}+2\sqrt{2}\sqrt{\varphi(n)}\right).$$
\end{thm}

A similar result was studied in a different context in \cite{AE15}. We note that our methods are completely different from that of \cite{AE15}. It is well known (see e.g. \cite[Sublemma 4.10.1]{PY98}) that $|{\rm Fix}_n| \sim e^{h(\sigma) n},$ that is, $\lim_{n \to \infty}|{\rm Fix}_n|/(e^{h(\sigma) n}) = 1.$ In fact, it is easy to see that $|{\rm Fix}_n|=\tr(A^n)=\lambda_1^n+\cdots+\lambda_s^n$, where $\lambda_i$'s are eigenvalues of $A$ with $\lambda_1 > |\lambda_i|$ for all $i \ne 1$. Thus, there exists $\delta>0$ such that 
$|{\rm Fix}_n|=\lambda_1^n(1+O(e^{-\delta n})),$ and since $h(\sigma)=\log \lambda_1$ we deduce that
$$h(\sigma)-\frac1n\log|{\rm Fix}_n| = O(e^{-\delta n}).$$ 
Hence, as a particular case of Theorem~\ref{thm:seqI} we obtain the effective equidistribution of periodic orbits:

\begin{thm} There exist constants $c,\delta >0$ such that for any Lipschitz function $f$ and $n \in \N$ we have
$$\left|\int f\, dm - \int f\, d\mu_{{\rm Fix}_n}\right| \le c \|f\|_\theta \, e^{-\delta n}.$$
\end{thm}

It is well known that repellers and Axiom A diffeomorphisms admit Markov partitions. We refer to \cite{B1, B2, S1, S2, Ka14} for more details. Consequently, we can realize a repeller $(J, T)$ and an Axiom A diffeomorphism $(\Omega,T)$ as a factor of a subshift of finite type. By abuse of notation let ${\rm Fix}_n$ denote the set of closed orbits $x$ with $T^nx=x.$ Using Theorem~\ref{thm:main} and standard arguments we can obtain the effective equidistribution statements of closed orbits.  

\begin{thm}\label{thm:RA} Let $(J,T)$ be a mixing repeller or $(\Omega(T),T)$ be a mixing Axiom A diffeomorphism. Let $\{I_n\}$ be a sequence of invariant sets with $I_n \subset {\rm Fix}_n$ and $\varphi(n):=h(T)-\frac1n \log|I_n| \to 0$ as $n \to \infty$ where $h(T)$ is the topological entropy. Then, for any Lipschitz function $f$ there exist constants $C(f)>0$ and $\theta \in (0,1)$ such that
$$|\int f\, dm - \int f\, d\mu_{I_n}| \le C(f) \left( \theta^n+2\sqrt{2}\sqrt{\varphi(n)}\right),$$
for any $n \in \N$. Moreover, there exists a constant $\delta >0$ such that for any Lipschitz function $f$ we have
$$\left|\int f\, dm - \int f\, d\mu_{{\rm Fix}_n}\right| \le C(f) e^{-\delta n}.$$
\end{thm}

We skip the proof and refer to \cite{Ka14} for more details about using the standard arguments. We note that defining suitable norms on $J$ or $\Omega$ one can make the dependence of $C(f)$ to $f$ precise. Equidistribution of closed orbits of expanding maps and hyperbolic diffeomorphisms were obtained by M.~Misiurewicz in \cite{Mi70} and by R.~Bowen in \cite{Bow71, B1}, respectively. Our results in Theorem~\ref{thm:RA} generalize and improve these results with exponential error terms.

We can also consider a subset ${\rm Fix}'_n$ of ${\rm Fix}_n$ consisting of primitive periodic orbits, that is, orbits with the least period $n$. When $T:\T^d \to \T^d$ is a linear hyperbolic automorphism of the $d$-torus, the equidistribution of ${\rm Fix}'_n$ was obtained in \cite{DI95}. From \cite[Proposition~2.3]{DI95} it follows that ${\rm Fix}'_n \sim \frac{e^{h(T) n}}{n}$, which implies that $h(T)-\frac{\log |{\rm Fix}'_n|}n = O(\frac{\log n}{n}).$ Thus, Theorem~\ref{thm:RA} in this special case gives the equidistribution of ${\rm Fix}'_n$ with the error term $O(\sqrt{\frac{\log n}{n}})$ improving \cite[Proposition~2.4]{DI95}.

\section{Proof of Theorem~\ref{thm:main}}

In this section we prove Theorem~\ref{thm:main}. We first note that the deduction of the second half of Theorem~\ref{thm:main} from \eqref{eqn:main} is standard as we sketch it now. It follows from \cite{B2, PP90} that if $f \in \mathcal F_\theta$ in $\Sigma_A$ then it is cohomologous to $f' \in \mathcal F_{\theta^{1/2}}$ with $f'(x)=f'(y)$ for all $x,y \in \Sigma_A$ satisfying $x_i=y_i$ for all $i \ge 0$. More precisely, there exist $f',u \in \mathcal F_{\theta^{1/2}}$ and a constant $C>0$ independent of $f$ such that $\|f'\|_{\theta/2} \le C \|f\|_\theta$ and $f+u \circ \sigma-u=f'.$ Then, $f'$ can be considered as a function in $\mathcal F_{\theta/2}^+$. Also, any invariant measure on $\Sigma_A$ can be considered as an invariant measure on $\Sigma_A^+$. Thus, using $\int f d\mu-\int f\, dm=\int f' d\mu-\int f' \,dm$ and \eqref{thm:main} we obtain the second half of the theorem. Thus, it suffices to prove Theorem~\ref{thm:main} for one-sided subshift.

We recall that $\xi_0^{N-1}=\bigvee_{i=0}^{N-1} \sigma^{-i} \xi$ where $N$ is a natural number or $N=\infty$. We first state properties of the measure $m$ of maximal entropy, known as Parry measure \cite{Parry}. This helps us to study the information function $I_m(\xi | \xi_1^{\infty})$ and see that $\int I_m(\xi | \xi_1^{\infty}) d\mu = h_m(\sigma)$ for any invariant measure $\mu$. We then use Pinsker inequality to relate the difference $ I_m(\xi | \xi_1^{\infty})- I_\mu(\xi | \xi_1^N)$ to the difference of entropies of partitions. What remains to do is to relate the difference of information functions to $\int f dm - \int f d\mu$ and this is done by constructing the sequence of functions $f_n=\mathcal L^n f $ using the transfer operator $\mathcal L$ for subshifts of finite type.

Let $A$ be an $s \times s$ irreducible and aperiodic transition matrix and $\lambda>0$ its largest eigenvalue. It follows from Perron-Frobenius theory cf. \cite[\S~0.9]{Walter}, that there are strictly positive left and right eigenvectors $(u_0,u_1, \dots,u_{s-1})$ and $(v_0,v_1,\dots,v_{s-1})$ respectively with $\sum_{i=0}^{s-1} u_i v_i=1$. We set $p_i=u_i v_i$ and $p_{ij}=a_{ij} v_j /\lambda v_i$. Then the Markov measure $  m$ given by the probability vector ${\bf p}=(p_0,p_1,\dots,p_{s-1})$ and the stochastic matrix $(p_{ij})$ is the unique measure of maximal entropy  \cite[Theorem~8.10]{Walter}. It is easy to see that for any admissible $(i_0,i_1,\dots,i_k)$, the $(k+1)$-cylinder set $C(i_0,i_1,\dots,i_k):=\{  x \in \Sigma_A^+ : x_0=i_0,\dots,x_k=i_k\}=\bigcap_{n=0}^k \sigma^{-n} C_{i_n}$ satisfies
\begin{equation}\label{eqn:volC}
  m(C(i_\ell,\dots,i_{\ell+k}))=\frac{u_{i_\ell} v_{i_{\ell+k}}}{\lambda^k}.
\end{equation}
For any partition $\zeta$ of $\Sigma_A^+$, let $[ x]_{\zeta}:=\bigcap_{  x \in B \in \zeta} B$ denote the atom of $\zeta$ containing $  x$ and $  m_{  x}^{\zeta}$ denote the conditional measure with respect to $\zeta$ supported on $[  x]_\zeta.$ For more information on conditional measures we refer to \cite[\S~5]{EW}. It follows from \eqref{eqn:volC} that
\begin{equation}\label{eqn:conditional}
  m_{  x}^{\xi_1^\infty} ([  x]_{\xi_0^{\infty}})= \lim_{N \to \infty}\frac{   m([  x]_{\bigvee_{i=0}^{N-1} \sigma^{-i}\xi})}{   m ([  x]_{\bigvee_{i=1}^{N-1} \sigma^{-i}\xi})}=\frac{u_{x_0}}{\lambda u_{x_1}}.
\end{equation}  
Thus, $ m_{  x}^{\xi_1^{\infty}} ([  x]_{\xi_0^{\infty}})$ is defined everywhere and for any $x \in \Sigma_A^+$ the information function $I_m$ satisfies
$$I_{  m}(\xi | \xi_1^{\infty})(  x)=-\log   m_{  x}^{\xi_1^{\infty}} ([x]_\xi)=\log \lambda+g(\sigma   x)-g(x),$$
where $g(  y)=\log u_{ y_0}.$ So, we immediately get

\begin{lem}\label{lem:iota}
For any $\sigma$-invariant probability measure $\mu$ on $\Sigma_A^+$, we have
$$\int I_{  m}(\xi | \xi_1^{\infty}) d\mu=h_{  m}(\sigma) =\log \lambda.$$
\end{lem}

We now state Pinsker inequality. Consider the $n$-dimensional simplex $\Delta_n$ of probability vectors $ q=(q_1,q_2,\dots,q_n)$. For a given $p \in \Delta_n$ with strictly positive entries we define the function
$$\phi_p: \Delta_n \to \R \text{ by } \phi_p(  q)=-\sum_{i=1}^n q_i \log \frac{p_i}{q_i},$$
with the convention $0 \log \frac{p_i}{0} = 0.$ Fix the norm $\|  q\|=\sum_i |q_i|$ on $\R^n$. We have \cite[Lemma 12.6.1]{CT}

 \begin{lem}[Pinsker Inequality]\label{lem:pinsker}
$\phi_p$ is nonnegative and has a unique $0$ at $  p$. Moreover, for any $  q\in \Delta_n$ we have $$\|  q-  p\| \le \sqrt{2\phi_p(  q)}.$$
\end{lem}

Let $  p(x),  q(x) \in \Delta_s$ be given by $p_i=p_i(  x)=  m_{  x}^{\xi_1^{\infty}}(C_i)$ and $q_i=q_i(  x)=\mu_{  x}^{\xi_1^N}(C_i)$. Then,
\begin{multline*}
\int (I_m(\xi | \xi_1^{\infty})(  y)-I_\mu(\xi |\xi_1^N)(  y)) d \mu_{  x}^{\xi_1^N}(  y)\\
=- \sum_{i=1}^s \mu_{  x}^{\xi_1^N}(C_i) \log \frac{  m_{  x}^{\xi_1^{\infty}}(C_i)}{\mu_{  x}^{\xi_1^N}(C_i)}=\phi_{p(x)}(q(  x)).
\end{multline*}
It is easy to see that $p_i=  m_{  x}^{\xi_1^{\infty}}(C_i)=0$ for some $i$ if and only if $C_i \cap [x]_{\xi_1^\infty}=\emptyset$ if and only if $C_i \cap [x]_{\xi_1^N}=\emptyset$ for any $N\in \N$. So, we must have $\mu_{  x}^{\xi_1^N}(C_i)=0$ in which case we simply drop the $i$-th term in the definition of $\phi_p.$ Now, applying Lemma~\ref{lem:iota} together with the fact $\int \int I_\mu(\xi | \xi_1^N) d \mu_{  x}^{\xi_1^N} d\mu(x)=H_\mu(\xi | \xi_1^N)$ we obtain
\begin{lem}\label{lem:intphi} For any invariant probability measure $\mu$ on $\Sigma_A^+$, we have
$$\int \phi_{p(x)}(q(x)) d\mu(x)=h_m(\sigma)-H_\mu(\xi | \xi_1^N).$$
\end{lem}

Now we are in a position to introduce the sequence $(f_n)_{n \ge 0}$ of functions using the transfer operator. Let $\mathcal L:L^1(\Sigma_A^+,\xi_0^{\infty},m) \to L^1(\Sigma_A^+,\xi_0^{\infty},m)$ denote the transfer operator given by 
$$\mathcal L f =\frac{d m_f \circ \sigma^{-1}}{d m} \text{ where } dm_f=f dm.$$
The following is classical (see e.g. \cite[Lemma~1.10 ]{B2} and \cite[Theorem 2.2]{PP90} ).

\begin{lem}\label{lem:fnplus1}
There exist constants $C>0$ and $\rho \in (0,1)$ such that for any Lipschitz function $g$ on $\Sigma_A^+$ with $\int g dm=0$ we have
$$\|\mathcal L^n g\|_\theta \le  C \rho^n \|g\|_\theta, \text{ for any } n \ge 0.$$
\end{lem}

We have

\begin{lem} \label{lem:nplus1nestim} 
For any $f \in \mathcal F_A^+$, any probability invariant measure $\mu$ on $\Sigma_A^+$, and $n,N \in \N$ we have
$$\left|\int f_{n+1}\,d\mu -\int f_n\, d\mu\right|\le \|f_n\|_\theta \left( \theta^{N+1}+\sqrt{2}\left(h_{  m}(\sigma)-H_\mu(\xi|\xi_1^N)\right)^{\frac12}\right),$$
where $f_{n}:=\mathcal L^n f=\mathcal L f_{n-1}.$
\end{lem}

\begin{proof}
It is easy to see that $(\mathcal L f) \circ \sigma = E_{  m} (f |  \xi_1^{\infty}).$ Hence, using $\sigma$-invariance of $\mu$ we have
\begin{equation}
\int f_{n+1}\,d\mu -\int f_n\, d\mu = \int E_{  m}(f_n|\xi_1^{\infty})d\mu- \int E_\mu(f_n|\xi_1^N) d\mu.
\end{equation}
Clearly $C_i \cap [x]_{\xi_1^{\infty}}=\{y^{(i)}=ix_1x_2\cdots\}$ or empty. In any case we have 
$$E_m(f_n|\xi_1^{\infty})(x)=\int f_n(y) \,dm_x^{\xi_1^{\infty}}(y)=\sum_{i \in \Lambda} f_n(y^{(i)}) m_x^{\xi_1^{\infty}}(C_i).$$
Also, for any $y \in C_i \cap [x]_{\xi_1^N}$ we have $d(y,y_i) \le \theta^{N+1}.$ Thus, for $\mu$-a.e.~$x$
\begin{multline*}
|E_\mu(f_n |\xi_1^N)(x) - \sum_{i \in \Lambda}f_n(y^{(i)})\mu_x^{\xi_1^N}(C_i)|\\
=| \sum_{i \in \Lambda}\int_{C_i} (f_n(y)-f_n(y^{(i)})\,d\mu_x^{\xi_1^N}(y)|\le \theta^{N+1}|f_n|_\theta.
\end{multline*}
Consequently, this gives
\begin{multline*}
| E_{  m}(f_n|\xi_1^\infty)(x)-  E_\mu(f_n|\xi_1^N)(x)|\\
\le \theta^{N+1}|f_n|_\theta + \sum_{i \in \Lambda}|f_n(y^{(i)})| |m_x^{\xi_1^\infty}(C_i)-\mu_x^{\xi_1^N}(C_i)|\\
\le  \theta^{N+1}|f_n|_\theta + |f_n|_\infty \|p(x)-q(x)\|.
\end{multline*}
Using Lemma~\ref{lem:pinsker} and Cauchy-Schwarz inequality we deduce  
\begin{align*}
|\int f_{n+1}\,d\mu -\int f_n\, d\mu| &\le \int ( \theta^{N+1}|f_n|_\theta + |f_n|_\infty \|p(x)-q(x)\|)\,d\mu\\
&\le  \theta^{N+1} |f_n|_\theta  + |f_n|_\infty\int \sqrt{2 \phi_{p(x)}(q(x))}\,d\mu(x)\\
&\le \theta^{N+1}|f_n|_\theta + |f_n|_\infty\sqrt{2 \int  \phi_{p(x)}(q(x))\,d\mu(x)}\\
&= \theta^{N+1}|f_n|_\theta + \sqrt{2}|f_n|_\infty\sqrt{h_m(\sigma) - H_\mu(\xi | \xi_1^N)}. 
\end{align*}
\end{proof}

We need one more lemma before we prove Theorem~\ref{thm:main}.

\begin{lem}\label{lem:ave} Let $(a_n)_{n \ge 0}$ be a decreasing sequence of nonnegative integers and set $A_n=\frac{a_0+a_1+\cdots+a_{n-1}}n$. Then, for any $n \in \N$ and $h \ge a_0$ we have 
$$ 2( h -A_n) \ge h-a_{\left\lfloor\frac{n}2\right\rfloor}.$$
\end{lem}
\begin{proof} It suffices to prove the lemma for $h=a_0$ and in this case the conclusion follows from the inequality
$$
A_n \le \frac{1}{n} \left(\left\lfloor\frac{n}2\right\rfloor a_0+(n-\left\lfloor\frac{n}2\right\rfloor)a_{\left\lfloor\frac{n}2\right\rfloor} \right)\le\frac{1}2(a_0+a_{\left\lfloor\frac{n}2\right\rfloor}).
$$
\end{proof}

\begin{proof}[Proof of Theorem~\ref{thm:main}]
It suffices to prove Theorem~\ref{thm:main} for Lipschitz functions $f$ with $\int f dm=0.$ As before we set $f_n=\mathcal L^n f$ for $n \ge 0$. From Lemma~\ref{lem:fnplus1} we see that $\int f_n \,d\mu$ converges to $0=\int f \,d  m $ which gives
\begin{align*}
\left| \int f d\mu - \int f d  m\right|=\lim_{n \to \infty} \left|\int f d\mu - \int f_n d\mu\right| \le \sum_{n=0}^\infty \left| \int f_{n+1} d\mu - \int f_n d\mu\right|.
\end{align*}
Now, using the estimate from Lemma~\ref{lem:nplus1nestim} together with Lemma~\ref{lem:fnplus1} we conclude
\begin{align*}
\left| \int f d\mu - \int f d  m\right| &\le \sum_{n=0}^\infty \|f_n\|_\theta \left( \theta^{N+1}+\sqrt{2}\left(h_{  m}(\sigma)-H_\mu(\xi|\xi_1^N)\right)^{\frac12}\right)\\
& \le \frac{C}{1-\rho} \|f\|_\theta \left( \theta^{N+1}+\sqrt{2}\left(h_{  m}(\sigma)-H_\mu(\xi|\xi_1^N)\right)^{\frac12}\right).
\end{align*}

Now, we consider how to replace $H_\mu(\xi|\xi_1^N)$ by $\frac{1}{N}H_\mu(\xi_0^{N-1})$. We know that 
\begin{equation}\label{eqn:cond}
\frac{1}{N}H_\mu(\xi_0^{N-1})=\frac{1}{N}H_\mu(\xi_0^{N-1})=\frac1N \sum_{n=0}^{N-1} H_\mu(\xi|\xi_1^n).
\end{equation}
It follows from Lemma~\ref{lem:pinsker} and Lemma~\ref{lem:intphi} that $H_\mu(\xi|\xi_1^n) \le h_m(\sigma)$ for any $n \in \N$ and in particular we have $\frac{1}{N}H_\mu(\xi_0^{N-1}) \le h_m(\sigma).$ Thus, applying Lemma~\ref{lem:ave} for $h=h_m(\sigma)$ and $a_n= H_\mu(\xi|\xi_1^n)$ we get
$$ 2( h_m(\sigma) -\frac{1}{N}H_\mu(\xi_0^{N-1})) \ge h_m(\sigma)-H_\mu(\xi|\bigvee_{i=1}^{\left\lfloor\frac{N}2\right\rfloor}\sigma^{-i}\xi).$$
Hence, for any $N \in \N$
$$\left| \int f d\mu - \int f d  m\right| \le \frac{C}{1-\rho} \|f\|_\theta \left( \theta^{\frac{N}2}+2\sqrt{2}\left( h_m(\sigma) -\frac{1}{N}H_\mu(\xi_0^{N-1})\right)^{\frac12} \right),$$
which finishes the proof.
\end{proof}

\end{document}